\documentclass[fleqn]{article}

\usepackage{amsmath,amssymb,amsthm,booktabs}

\theoremstyle{definition}
\newtheorem{defi}{Definition}[section]

\newtheorem{thm}[defi]{Theorem}
\newtheorem{lem}[defi]{Lemma}
\newtheorem{cor}[defi]{Corollary}
\newtheorem{cla}[defi]{Claim}
\newtheorem{rem}[defi]{Remark}
\newtheorem{prob}[defi]{Problem}

\newcommand{\BL}{\mathsf{BL}}
\newcommand{\CW}{\mathsf{CW}}
\newcommand{\FI}{\mathsf{FI}}
\newcommand{\FH}{\mathsf{FH}}
\newcommand{\FIFD}{\mathsf{FIFD}}
\newcommand{\GL}{\mathbf{GL}}
\newcommand{\PA}{\mathbf{PA}}
\newcommand{\NQGL}{\mathbf{NQGL}}
\newcommand{\QK}{\mathbf{QK}}
\newcommand{\QKf}{\mathbf{QK4}}
\newcommand{\QGL}{\mathbf{QGL}}

\newcommand{\prove}{\vdash}

\newcommand{\Lob}{L\"{o}b}
\newcommand{\Smor}{Smory\'{n}ski}

\title{Fixed-point properties for predicate modal logics}
\author{Sohei Iwata and Taishi Kurahashi}
\date{}

\begin{document}

\maketitle

\begin{abstract}
It is well known that the propositional modal logic $\GL$ of provability satisfies the de Jongh-Sambin fixed-point property. On the other hand, Montagna showed that the predicate modal system $\QGL$, which is the natural variant of $\GL$, loses the fixed-point property. 
In this paper, we discuss some versions of the fixed-point property for predicate modal logics. 
First, we prove that several extensions of $\QGL$ including $\NQGL$ do not have the fixed-point property. 
Secondly, we prove the fixed-point theorem for the logic $\QK + \Box^{n+1} \bot$. 
As a consequence, we obtain that the class $\FH$ of Kripke frames which are transitive and finite height satisfies the fixed-point property locally. 
We also show the failure of the Craig interpolation property for $\NQGL$.
Finally, we give a sufficient condition for formulas to have a fixed-point in $\QGL$.
\end{abstract}

\section{Introduction}

The propositional modal system $\GL$ is obtained from the smallest normal modal logic $\mathbf{K}$ by adding the axiom schema $\Box ( \Box A \to A) \to \Box A$. The modal system $\GL$ is well known as \textit{the logic of provability}, since it has the connection with arithmetical theories, for instance, Peano Arithmetic $\PA$ (cf. Solovay \cite{Sol76}).

One of the fundamental results about the logic of provability is the de Jongh-Sambin fixed-point theorem which is a natural counterpart of the fixed-point lemma in arithmetic (cf.~\cite{Sam76}). 
Let $A(p)$ be a propositional modal formula. We say $A(p)$ is modalized in $p$ if all occurrences of the propositional variable $p$ in $A(p)$ are within the scope of the modal operator. The de Jongh-Sambin fixed-point theorem states that if $A(p)$ is modalized in $p$, then there is a propositional modal formula $B$ containing only propositional variables occurring in $A(p)$, not containing $p$, and such that $\GL \prove B \leftrightarrow A(B)$.
The fixed-point theorem also holds for the logic $\mathbf{K} + \Box^{n+1} \bot$ which is due to Sacchetti \cite{Sac99}.

It is natural to extend these studies to predicate modal logic. 
However, the situation of the predicate logic of provability is quite complex and most of the properties for $\GL$ do not hold for the predicate modal system $\QGL$ which is the natural extension of $\GL$. In particular, Montagna \cite{Mon84} showed that $\QGL$ does not satisfy any of the Kripke completeness, the arithmetical completeness, and the de Jongh-Sambin fixed-point property. 

On the other hand, there is a room for investigations of the fixed-point property in predicate modal logic. 
Although the logic $\QGL$ is a natural candidate of an extension of $\GL$, it is not the only one. 
For example, recently Tanaka \cite{Tana18} introduced a new predicate modal logic $\NQGL$ which is strictly stronger than $\QGL$.
The fixed-point theorem may hold in one of these natural extensions of $\QGL$. 
Also it has not been known whether the fixed-point theorem for $\mathbf{K}+\Box^{n+1} \bot$ can be extended to predicate modal logic. 
In this paper, we investigate some versions of the fixed-point property for predicate modal logics. 

In Section 2, we introduce predicate modal logics and Kripke semantics, and define the following five classes of Kripke frames in which all theorems of $\QGL$ are valid: 
$\CW$ (the class of transitive and conversely well-founded frames), $\BL$ (the class of transitive frames of which is bounded length), $\FH$ (the class of transitive frames with finite height), $\FI$ (the class of finite transitive irreflexive frames), and $\FIFD$ (the class of finite transitive irreflexive frames of which domains are finite).
The class $\BL$ is introduced by Tanaka \cite{Tana18}, and he showed that $\NQGL$ is Kripke complete with respect to $\BL$.
The class $\FIFD$ was investigated by Artemov and Japaridze \cite{AJ90}. 

We investigate two semantical fixed-point properties for classes of frames, that is, the fixed-point property and the local fixed-point property. It follows that, by Montagna's proof, the classes $\CW$ and $\BL$ do not enjoy neither the local fixed-point property nor the fixed-point property. In Sections 3 and 4, we discuss whether the classes $\FH$, $\FI$ and $\FIFD$ enjoy these two properties. 
In Section 3, we prove that the classes $\FH$, $\FI$ and $\FIFD$ do not enjoy the fixed-point property. 
In Section 4, we prove the fixed-point theorem for the predicate modal logic $\QK + \Box^{n+1} \bot$. 
We stress that our proof provides an algorithm for calculating fixed-points in these logics. 
As a consequence, we show that the classes $\FH$, $\FI$ and $\FIFD$ enjoy the local fixed-point property. 
This shows that the logics determined by these classes are consistent with the fixed-point property (cf. Sacchetti \cite{Sac99}).
Table 1 summarizes the situation of these semantical fixed-point properties. 

In Section 5, we prove that $\NQGL$ does not enjoy the Craig interpolation property. 
This is a consequence of our result proved in Section 4. 
As mentioned above, the de Jongh-Sambin fixed-point theorem does not hold for $\QGL$. 
Although there is a possibility that the fixed-point theorem holds for some classes of formulas, it has not been known sufficient (and necessary) conditions for a formula to have a fixed-point in $\QGL$. 
In Section 6, we argue a sufficient condition for a formula $A(p)$ to have a fixed-point in $\QGL$. We prove that, if $A(p)$ is a Boolean combination of $\Sigma$-formulas, then $A(p)$ has a fixed-point in $\QGL$.

\begin{table}[ht]
	\begin{center}
		\caption{Five classes and the fixed-point properties}
		\begin{tabular}{ccc}
		\toprule class  & FPP & localFPP \\ \midrule
		$\FIFD$ & No & Yes \\
		$\FI$ & No & Yes \\
		$\FH$ & No & Yes \\ 
		$\BL$ & No & No \\
		$\CW$ & No & No \\ \bottomrule
		\end{tabular}
	\end{center}
\end{table}

\section{Preliminaries}

\subsection{Predicate modal logic and its Kripke semantics}\label{Sec2-1}

The language of predicate modal logic $\mathcal{L}$ consists of countably many variables $u, v, \ldots$, etc., Boolean constants $\top, \bot$, Boolean connectives $\neg, \to$, quantifier $\forall$, and countably many predicate symbols for each arity (denoted by $P, Q, \ldots$ etc.). An $\mathcal{L}$-formula $A$ is constructed as the following manner:
\begin{equation*}
A ::= \top \mid \bot \mid P(u_1, \ldots, u_n) \mid \neg A \mid A \to A \mid \forall u A \mid \Box A
\end{equation*}
where $P$ is an $n$-ary predicate symbol, and $u_1, \ldots, u_n, u$ are variables.
Let $\Box^{n} A :\equiv \overbrace{\Box \cdots \Box}^{n} A$, and $\boxdot A :\equiv \Box A \land A$.

Boolean constants $\top$ and $\bot$, and $\mathcal{L}$-formulas of the form $P(u_1, \ldots, u_n)$ are called \textit{atomic formulas}. We put 
\begin{gather*}
A \lor B :\equiv \neg A \to B, \quad A \land B :\equiv \neg ( A \to \neg B), \quad A \leftrightarrow B :\equiv (A \to B) \land (B \to A), \\
\exists u A :\equiv \neg \forall u \neg A, \quad \Diamond A :\equiv \neg \Box \neg A.
\end{gather*}

Free variables and bound variables are naturally defined. We say $A$ is an $\mathcal{L}$-sentence if $A$ is an $\mathcal{L}$-formula with no free variables.

The predicate modal system $\QK$ consists of the following axioms and rules:
\begin{description}
\item[Ax1] All instances of axioms of predicate logic in the language $\mathcal{L}$;
\item[Ax2] $\Box ( A \to B ) \to ( \Box A \to \Box B )$;
\item[R1] $A,\ A \to B / B$ (\textit{modus ponens});
\item[R2] $A / \Box A$ (\textit{necessitation}).
\end{description}

The predicate modal systems $\QKf$ and $\QGL$ are obtained from $\QK$ by adding the following axioms $\mathbf{4}$, and \textbf{\Lob}, respectively.
\begin{description}
\item[4] $\Box A \to \Box \Box A$;
\item[\Lob] $\Box ( \Box A \to A ) \to \Box A$. 
\end{description}

Recall that $\QK \subseteq \QKf \subseteq \QGL$.

\begin{defi}[Kripke frames]
A \textit{Kripke frame} $\mathcal{F}$ is a triple $\langle W, \prec, {\{D_w \}}_{w \in W} \rangle$ where:
	\begin{itemize}
	\item $W$ is a non-empty set;
	\item $\prec$ is a binary relation on $W$;
	\item Each $D_w$ is a non-empty set, and if $w \prec w'$, then $D_{w} \subseteq D_{w'}$.
	\end{itemize}
\end{defi}

\begin{defi}[Interpretations and Kripke models]
Let $\mathcal{F} = \langle W, \prec , {\{ D_w \}}_{w \in W} \rangle$ be a Kripke frame. \textit{An interpretation of $\mathcal{F}$} is a mapping $\Vdash$ which assigns each pair $\langle w, P \rangle$, where $w \in W$ and $P$ is an $n$-ary predicate symbol, into an $n$-ary relation on $D_w$. We write $w \Vdash P(a_1, \ldots, a_n)$ if $( a_1, \ldots, a_n )$ is a member of $\Vdash \langle w, P \rangle$. \textit{A Kripke model} $\mathcal{M}$ is a pair $\langle \mathcal{F},\Vdash \rangle$ where $\mathcal{F}$ is a Kripke frame and $\Vdash$ is an interpretation of $\mathcal{F}$.
\end{defi}

\begin{defi}[Truth value]
Let $\mathcal{M} = \langle W, \prec, {\{ D_w \}}_{w \in W}, \Vdash \rangle$ be a Kripke model, and $A$ be an $\mathcal{L}$-sentence with parameters from $D_w$ for some $w \in W$. \textit{The truth value of $A$ in $w$} (We write $\mathcal{M}, w \models A$ if $A$ is true in $w$) is inductively defined as follows:
	\begin{itemize}
	\item $\mathcal{M}, w \models \top$ and $\mathcal{M}, w \not \models \bot$, for every $w \in W$;
	\item $\mathcal{M}, w \models P(a_1, \ldots a_n)$ iff $w \Vdash P(a_1, \ldots a_n)$;
	\item $\mathcal{M}, w \models \neg A$ iff $\mathcal{M}, w \not \models A$;
	\item $\mathcal{M}, w \models A \to B$ iff $\mathcal{M}, w \not\models A$ or $\mathcal{M}, w \models B$;
	\item $\mathcal{M}, w \models \forall u A(u)$ iff $\mathcal{M}, w \models A(a)$ for every $a \in D_w$;
	\item $\mathcal{M}, w \models \Box A$ iff for any $v \in W$, if $w \prec v$, then $\mathcal{M}, v \models A$.
	\end{itemize}
\end{defi}

\begin{defi}[Validity]
Let $\mathcal{M}$ be a Kripke model and $A$ be an $\mathcal{L}$-sentence. We say \textit{$A$ is valid in $\mathcal{M}$} (write $\mathcal{M} \models A$) if for every $w \in W$, $\mathcal{M}, w \models A$.

Let $\mathcal{F}$ be a Kripke frame and $A$ be an $\mathcal{L}$-sentence. We say \textit{$A$ is valid in $\mathcal{F}$} (write $\mathcal{F} \models A$) if for any interpretation $\Vdash$ of $\mathcal{F}$, $A$ is valid in $\mathcal{M} = \langle \mathcal{F}, \Vdash \rangle$.
\end{defi}

Validity of an $\mathcal{L}$-formula $A$ is defined by the validity of the universal closure of $A$.

Next we specify several classes of Kripke frames. Let $\mathcal{F} = \langle W, \prec , \{ D_w \}_{w \in W} \rangle$ be a Kripke frame.
We say \textit{$\mathcal{F}$ is finite} if $W$ is finite. A Kripke frame $\mathcal{F}$ is \textit{conversely well-founded} if there is no countably infinite sequence $(w_i)_{i < \omega}$ of worlds of $W$ satisfying $w_i \prec w_{i+1}$ for each $i < \omega$.

Suppose that $\mathcal{F}$ is conversely well-founded. For each $w \in W$, the \textit{height of $w$} (write $h(w)$) is defined inductively by: 
$$h(w) := \sup \{ h(v) +1 : w \prec v \} .$$
(In particular, $\sup \emptyset = 0$.) 
A Kripke frame $\mathcal{F}$ is \textit{of bounded length} if for any $w \in W$, $h(w)$ is finite.
For a Kripke frame $\mathcal{F}$, \textit{the height of $\mathcal{F}$} is defined by $\sup \{ h(w) : w \in W \}$, and $\mathcal{F}$ is said to be \textit{finite height} if $h(\mathcal{F})$ is finite. 

We define the following five classes of Kripke frames:
\begin{enumerate}
\item $\CW := \{ \mathcal{F} \mid \mathcal{F}$ is transitive and conversely well-founded$\}$;
\item $\BL := \{ \mathcal{F} \mid \mathcal{F}$ is transitive and of bounded length$\}$;
\item $\FH := \{ \mathcal{F} \mid \mathcal{F}$ is transitive and finite height$\}$;
\item $\FI := \{\mathcal{F} \mid \mathcal{F}$ is finite, transitive and irreflexive$\}$;
\item $\FIFD := \{ \mathcal{F} \mid \mathcal{F}$ is finite, transitive and irreflexive, and for every $w \in  W$, $D_{w}$ is finite$\}$.
\end{enumerate}

For a class $\mathsf{C}$ of Kripke frames, $\mathbf{MQ}(\mathsf{C})$ denotes the set of all $\mathcal{L}$-formulas which are valid in any $\mathcal{F}$ in $\mathsf{C}$. It is easy to show that $\QGL \subseteq \mathbf{MQ}(\mathsf{CW})$. Since $\mathsf{FIFD} \subseteq \FI \subseteq \FH \subseteq \BL \subseteq \CW$, we obtain
\[
	\QGL \subseteq \mathbf{MQ}(\CW) \subseteq \mathbf{MQ}(\BL) \subseteq \mathbf{MQ}(\FH) \subseteq \mathbf{MQ}(\FI) \subseteq \mathbf{MQ}(\FIFD).
\] 
The class $\BL$ is introduced by Tanaka \cite{Tana18}. 

It is easy to show $\mathbf{MQ}(\BL) = \mathbf{MQ}(\FH)$. For, if $A \not\in \mathbf{MQ}(\BL)$, then there exsist a model $\mathcal{M} = \langle W, \prec, {\{ D_w \}}_{w \in W}, \Vdash \rangle$ and $w \in W$ such that $\langle W, \prec, {\{ D_w \}}_{w \in W} \rangle \in \BL$ and $\mathcal{M} ,w \not\models A$. Let $\mathcal{M}^\ast$ be the generated submodel of $\mathcal{M}$ by $w$. Then the frame of $\mathcal{M}^\ast$ is finite height, and $\mathcal{M}^\ast, w \not\models A$. Hence, $A \not\in \mathbf{MQ}(\FH)$.

Tanaka also introduced the modal proof system $\NQGL$ which has an infinitary inference rule, and showed that $\NQGL$ is Kripke complete with respect to $\BL$.

\begin{defi}[The system $\NQGL$, \cite{Tana18}]
The system $\NQGL$ is obtained from $\QKf$ by adding the following rule:
	\begin{description}
	\item[BL] If $\prove \Box^{n+1} \bot \to A$ for all natural numbers $n$, then $\prove A$. 
	\end{description}
\end{defi}
\begin{thm}[Tanaka \cite{Tana18}]\label{Tana}
$\NQGL = \mathbf{MQ}(\BL) \left(= \mathbf{MQ}(\FH) \right)$.
\end{thm}

By Theorem \ref{Tana}, we obtain $\QGL \subseteq \NQGL$.

\subsection{Fixed point properties}\label{Sec22}

The fixed-point theorem was originally proved by de Jongh and Sambin \cite{Sam76} for the propositional logic $\mathbf{GL}$ independently. 
In \cite{Sac99} Sacchetti proved the fixed-point theorem for the logic $\mathbf{K} + \Box^{n+1} \bot$. 
Let $A(p)$ be a propositional modal formula containing occurrences of $p$. We say $A(p)$ is \textit{modalized in $p$} if every occurrence of $p$ in $A(p)$ is in the scope of modal operators. For a propositional modal formula $B$, $A(B)$ denotes the one obtained from $A$ by substituting $B$ for all occurrences $p$ in $A$. To summarize the results, the fixed-point theorems are described as follows.

\begin{thm}[Fixed-point theorem (de Jongh, Sambin \cite{Sam76}, and Sacchetti \cite{Sac99})]\label{FPT}
Suppose that $\mathbf{L}$ is either $\GL$ or $\mathbf{K} + \Box^{n+1} \bot$. If $A(p)$ is modalized in $p$, then there is a formula $B$ containing only propositional variables occurring in $A(p)$, not containing $p$, and such that $\mathbf{L} \prove B \leftrightarrow A(B)$.
\end{thm}
We call such a $B$ \textit{a fixed-point of $A(p)$ in $\mathbf{L}$}.

To describe the fixed-point properties for predicate modal logic, we need an auxiliary propositional variable to specify where to substitute fixed-points in predicate modal formulas. 
For this purpose, we define the following language $\mathcal{L}'$.
The language $\mathcal{L}'$ consists of $\mathcal{L}$ and one certain fixed propositional variable $p$. An $\mathcal{L}'$-formula $A$ is constructed as the following manner:
\begin{equation*}
 A ::= \top \mid \bot \mid p \mid P(u_1, \ldots, u_n) \mid \neg A \mid A \to A \mid \forall u A \mid \Box A
\end{equation*}

Montagna \cite{Mon84} showed that the predicate version of Theorem \ref{FPT} does not hold in $\mathbf{QGL}$.

\begin{thm}[Montagna \cite{Mon84}]\label{QFPT}
Let $A(p)$ be the $\mathcal{L}'$-sentence $\forall u \exists v \Box \left( p \to P(u, v) \right)$. Then $A(p)$ has no fixed-points in $\QGL$, that is, for any $\mathcal{L}$-sentence $B$ containing only the predicate symbol $P$, $\mathbf{QGL} \nvdash B \leftrightarrow A(B)$.
\end{thm}

Here we define two semantical fixed-point properties for classes of frames. 

\begin{defi}
Let $\mathsf{C}$ be a class of Kripke frames.
	\begin{enumerate}
	\item The class $\mathsf{C}$ has \textit{the fixed-point property} if for any $\mathcal{L}'$-formula $A(p)$ which is modalized in $p$,
		there exists an $\mathcal{L}$-formula $B$ such that:
		\begin{enumerate}
		\item The formula $B$ contains only predicate symbols occurring in $A$;
		\item For any Kripke frame $\mathcal{F}$ in $\mathsf{C}$, $\mathcal{F} \models B \leftrightarrow A(B)$.
		\end{enumerate}
	\item The class $\mathsf{C}$ has \textit{the local fixed-point property} if for any $\mathcal{L}'$-formula $A(p)$ which is modalized in $p$, and for any Kripke frame $\mathcal{F}$ in $\mathsf{C}$, there exists an $\mathcal{L}$-formula $B$ such that:
		\begin{enumerate}
		\item The formula $B$ contains only predicate symbols occurring in $A$;
		\item $\mathcal{F} \models B \leftrightarrow A(B)$.
		\end{enumerate}
	\end{enumerate}
\end{defi}

Clearly if $\mathsf{C}$ has the fixed-point property, then $\mathsf{C}$ has the local fixed-point property. 
Montagna proved Theorem \ref{QFPT} by constructing a Kripke model $\mathcal{M}$ in $\BL$ such that for any $\mathcal{L}$-sentence $B$ containing only $P$, the formula $B \leftrightarrow A(B)$ is not valid in $\mathcal{M}$.
Thus we obtain the following corollary:

\begin{cor}\label{C2}\leavevmode
	\begin{enumerate}
	\item The classes $\CW$ and $\BL$ have neither the local fixed-point property, nor the fixed-point property.
	\item The fixed-point theorem for $\NQGL$ does not hold.
	\end{enumerate} 
\end{cor}

The second clause of Corollary \ref{C2} immediately follows from the first clause and Theorem \ref{Tana}. 

\subsection{The substitution lemma}

The following substitution lemma will be used in Sections 5 and 6.

\begin{lem}[Substitution lemma]\label{L16}
Let $A(p)$ be any $\mathcal{L}'$-formula.
Let $F$ and $G$ be $\mathcal{L}$-formulas containing no free variables which are bounded in $A(p)$.
Then $\QKf \prove \boxdot (F \leftrightarrow G) \to \left( A(F) \leftrightarrow A(G) \right)$. Moreover,
if $A(p)$ is modalized in $p$, then
$\QKf \prove \Box ( F \leftrightarrow G) \to \left( A(F) \leftrightarrow A(G) \right)$.
\end{lem}
	
\begin{proof}
Induction on the construction of $A(p)$.
	\begin{itemize}
	\item If $A(p)$ does not contain $p$, then Lemma trivially holds.
	\item Assume $A(p) \equiv p$. Then $A(F) \equiv F$ and $A(G) \equiv G$, and thus Lemma holds.
	\item The cases $A(p) \equiv \neg B(p)$ and $A(p) \equiv B(p) \to C(p)$ are clear.
	\item Assume $A(p) \equiv \forall u B(p)$ and Lemma holds for $B(p)$. If $F$ and $G$ contain no free
	variables which are bounded in $A(p)$, then every free variable of $F$ and $G$ is not equal to $u$, and 
	hence is not bounded in $B(p)$. By the induction hypothesis, 
	$\QKf \prove \boxdot (F \leftrightarrow G) \to \left( B(F) \leftrightarrow B(G) \right)$. Since $u$
	does not occur freely in $F$ and $G$, we have
	$\QKf \prove \boxdot (F \leftrightarrow G) \to \forall u \left( B(F) \leftrightarrow B(G) \right)$.
	Distributing $\forall$, we conclude
	$\QKf \prove \boxdot (F \leftrightarrow G) \to
	\left( \forall u B(F) \leftrightarrow \forall u B(G) \right)$.
	(If $A(p)$ is modalized in $p$, then so is $B(p)$.
	By the induction hypothesis,
	$\QKf \prove \Box (F \leftrightarrow G) \to \left( B(F) \leftrightarrow B(G) \right)$.
	Applying a similar argument, we conclude $\QKf \prove \Box (F \leftrightarrow G) \to \left( \forall u B(F)
	\leftrightarrow \forall u B(G) \right)$.)
	\item Assume $A(p) \equiv \Box B(p)$ and Lemma holds for $B(p)$. By the induction hypothesis,
	$\QKf \prove \boxdot (F \leftrightarrow G) \to \left( B(F) \leftrightarrow B(G) \right)$.
	By the derivation of $\QK$,
	$\QKf \prove \Box \boxdot (F \leftrightarrow G) \to
	\left( \Box B(F) \leftrightarrow \Box B(G) \right)$.
	Recall that $\QKf \prove \Box E \to \Box \boxdot E$ for any $E$. Thus we conclude
	$\QKf \prove \Box (F \leftrightarrow G) \to
	\left( \Box B(F) \leftrightarrow \Box B(G) \right)$.
	\end{itemize}
\end{proof}

\section{Failure of the fixed-point property for $\FIFD$}

In this section, we prove that the class $\FIFD$ dos not enjoy the fixed-point property. 
As a consequence, we obtain that the classes $\FH$ and $\FI$ also do not have the fixed-point property. 

In our proof, we borrow an idea from the following Smory\'nski's improvement of Montagna's theorem (Theorem \ref{QFPT}). 

\begin{thm}[\Smor\ \cite{Smo87}]\leavevmode \label{SmT}
The $\mathcal{L}'$-formula $\forall u \Box ( p \to P ( u ) )$ has no fixed-points in $\QGL$. 
\end{thm}

The details of the proof of Theorem \ref{SmT} is as follows. 
Let $\mathbb{N}$ be the set of all natural numbers, and $\mathcal{M}_S := \langle W, \prec, {\{ D_n \}}_{n \in W}, \Vdash \rangle$ where
	\begin{itemize}
	\item $W := \mathbb{N}$;
	\item $m \prec n : \Leftrightarrow n < m$;
	\item $D_n := \{ m \in \mathbb{N} \mid m \geq n \}$;
	\item $n \Vdash P(m) : \Leftrightarrow m \neq n+1$.
	\end{itemize}

The Kripke frame $\langle W, \prec, {\{ D_n \}}_{n \in W} \rangle$ is a member of $\BL$. The following claim holds for $\mathcal{M}_S$.

\begin{cla}[\Smor\ \cite{Smo87}]\label{Sm}
Let $A$ be an $\mathcal{L}$-sentence containing only the predicate symbol $P$. Then the set $\{ n \in \mathbb{N} \mid \mathcal{M}_S, n \models A \}$ is either finite or co-finite.
\end{cla}

Using this fact, \Smor \ showed that for any $\mathcal{L}$-sentence $B$ containing only $P$, the formula $B \leftrightarrow A(B)$ is not valid in $\mathcal{M}_S$, and hence $\QGL \nvdash B \leftrightarrow A(B)$.

First, we prove the following lemma concerning Smory\'nski's model $\mathcal{M}_S$.

\begin{lem}\label{L1}
Let $n \in \mathbb{N}$ and $A(u)$ be an $\mathcal{L}$-formula with parameters from $D_n$ containing only the predicate symbol $P$. Then for any $m_1, m_2 \geq n+2$, 
	\begin{equation*}
	\mathcal{M}_S, n \models A(m_1) \leftrightarrow A(m_2).
	\end{equation*}
\end{lem}

\begin{proof}
Induction on the construction of $A(u)$.
	\begin{itemize}
	\item The cases $A(u) \equiv \top$ and $A(u) \equiv \bot$ are trivial.
	\item Assume $A(u) \equiv P(u)$. Then by the definition of $\Vdash$, for any $m_1, m_2 \geq n+2$,
	$\mathcal{M}_S, n \models P(m_1)$ and  $\mathcal{M}_S, n \models P(m_2)$.
	\item The cases $A(u) \equiv \neg B(u)$ and $A(u) \equiv B(u) \to C(u)$ are clear by the induction hypothesis.
	\item Assume $A(u) \equiv \forall v B(u, v)$. Then
		\begin{align*}
		\mathcal{M}_S, n \models \forall v B(m_1, v) & \iff \mathcal{M}_S, n \models B(m_1, m')
		\text{ for any } m' \in D_n, \\
		\iff & \mathcal{M}_S, n \models B(m_2, m') \text{ for any } m' \in D_n, \tag{I.H.}\\
		\iff & \mathcal{M}_S, n \models \forall v B(m_2, v). 
		\end{align*}
	\item Assume $A(u) \equiv \Box B(u)$. Then
		\begin{align*}
		\mathcal{M}_S, n \models \Box B(m_1) & \iff \mathcal{M}_S, k \models B(m_1) \text{ for any } k < n.
		\end{align*}
	By $D_n \subseteq D_k$ for any $k < n$, $B(u)$ is an $\mathcal{L}$-formula with parameters from $D_k$. By the induction hypothesis (note that $k+2 < n+2 \leq m_1, m_2$),
		\begin{align*}
		\mathcal{M}_S, k \models B(m_1) \text{ for any } k < n & \iff 
		\mathcal{M}_S, k \models B(m_2) \text{ for any } k < n, \\
		 & \iff \mathcal{M}_S, n \models \Box B (m_2).
		\end{align*}
	\end{itemize}
\end{proof}

Next, we define Kripke models which are finitizations of Smory\'nski's model $\mathcal{M}_S$. 
For each $k \in \mathbb{N}$, we define $\mathcal{M}_k: = \langle W_k, \prec_k, {\{ D_n^k \}}_{n \in W_k}, \Vdash_k \rangle$ where
\begin{itemize}
\item $W_k := \{ 0, 1, \ldots, k \}$;
\item $m \prec_k n :\Leftrightarrow m \prec n (\Leftrightarrow n < m)$;
\item $D_n^k := \{ n, n+1, \ldots, k+2 \}$;
\item $n \Vdash_k P(m) :\Leftrightarrow n \Vdash P(m) (\Leftrightarrow m \neq n+1)$.
\end{itemize}

For each $k \in \mathbb{N}$, the frame $\langle W_k, \prec_k, {\{ D_n^k \}}_{n \in W_k} \rangle$ belongs to $\FIFD$.  

\begin{lem}\label{L2}
Fix $k \in \mathbb{N}$. For any $n \leq k$ and $\mathcal{L}$-sentence $A$ with parameters from $D_n^k$ containing only $P$,
	\begin{equation*}
	\mathcal{M}_S, n \models A \iff \mathcal{M}_k, n \models_k A.
	\end{equation*}
\end{lem}

\begin{proof}
Induction on the construction of $A$.
	\begin{itemize}
	\item The cases $A \equiv \top$ and $A \equiv \bot$ are trivial.
	\item Assume $A \equiv P(m)$ for some $m \in D_n^k$. By the definition of $\Vdash_k$, $\mathcal{M}_S, n \models P(m) \Leftrightarrow \mathcal{M}_k, n \models P(m)$.
	\item The cases for $A \equiv \neg B$, and $A \equiv B \lor C$ are clear by the induction hypothesis.
	\item Assume $A \equiv \forall u B(u)$. Then
		\begin{align*}
		\mathcal{M}_S, n \models \forall u B(u) &\iff \mathcal{M}_S, n \models B(m) \text{ for all } m \in D_n,  \\
		&\iff \mathcal{M}_S, n \models B(n), \ldots, \mathcal{M}_S, n \models B(k+1) \text{ and}  \\
		&\hspace{10.5mm} \mathcal{M}_S, n \models B(m) \text{ for all } m \geq k+2.  \tag{$\star$}
		\end{align*}
	By Lemma \ref{L1}, the statement $(\star)$ is equivalent to $\mathcal{M}_S, n \models B(k+2)$. Thus
		\begin{align*}
		\mathcal{M}_S, n \models \forall u B(u) &\iff \mathcal{M}_S, n \models B(n), \ldots, \mathcal{M}_S, n \models B(k+2), \\
		& \iff \mathcal{M}_k, n \models B(n), \ldots, \mathcal{M}_k, n \models B(k+2), \tag{I.H.} \\
		& \iff \mathcal{M}_k, n \models \forall u B(u). 
		\end{align*}
	\item If $A \equiv \Box B$, then
		\begin{equation*}
		\mathcal{M}_S, n \models \Box B  \iff \mathcal{M}_S, m \models B \text{ for all }m < n.
		\end{equation*}
	Since $D_n^k \subseteq D_m^k$ for any $m < n$, $B$ is an $\mathcal{L}$-sentence with parameters from
	$\bigcap_{m < n} D_m^k$, and hence
		\begin{align*}
		\mathcal{M}_S, m \models B \text{ for all }m < n & \iff \mathcal{M}_k, m \models B \text{ for all } m < n,  \tag{I.H.} \\
		& \iff  \mathcal{M}_k, n \models \Box B. 
		\end{align*}
	\end{itemize}
\end{proof}

\begin{lem}\label{L3}
Fix $k \in \mathbb{N}$. For any $\mathcal{L}$-sentence $A$, if $\mathcal{M}_k \models A \leftrightarrow \forall u \Box \left( A \to P(u) \right)$, then for any $n \leq k$,
	\begin{center}
	$\mathcal{M}_k , n \models A$ $\iff$ $n$ is even.
	\end{center}
\end{lem}

\begin{proof}
Induction on $n$.

Assume $n = 0$. Since $\mathcal{M}_k, 0 \models \Box ( A \to P(m) )$ for any $m \in D_0^k$, we have $\mathcal{M}_k, 0 \models \forall u \Box ( A \to P(u) )$. By the assumption, $\mathcal{M}_k, 0 \models A$.
	
(Inductive case) Assume Lemma holds for $m < n$.
	\begin{itemize}
	\item[($\Rightarrow$)]
	Suppose that $n$ is odd. Since $\mathcal{M}_k, {n-1} \models A$ and
	$\mathcal{M}_k, {n-1} \not \models P(n)$, we have $\mathcal{M}_k, n \not \models \Box ( A \to P(n) )$.
	This implies $\mathcal{M}_k, n \not \models \forall u \Box (A \to P(u) )$. By the assumption,
	$\mathcal{M}_k, n \not \models A$.
	\item[($\Leftarrow$)]
	Suppose that $n\neq 0$ and $n$ is even. We claim that
	$\mathcal{M}_k, n \models \Box ( A \to P(m) )$ for any $m \in D_n^k$. Take an arbitrary $l < n$.
	If $l < n-1$, then for every $m \in D_n^k$, $l+1 < n \leq m$, and hence $m \neq l+1$. Therefore for every $m \in D_n^k$,
	$\mathcal{M}_k, l \models P(m)$. This implies that for every $l < n-1$ and $m \in D_n^k$, $\mathcal{M}_k, l \models A \to P(m)$.

	If $l = n-1$, then $l$ is odd. By the induction hypothesis, $\mathcal{M}_k, l \not \models A$, and hence
	for every $m \in D_n^k$, $\mathcal{M}_k, l \models A \to P(m)$.
	
	We obtain that for every $l < n$ and $m \in D_n^k$, $\mathcal{M}_k, l \models A \to P(m)$, and hence the claim is verified.
	Thus, $\mathcal{M}_k, n \models \forall u \Box (A \to P(u) )$. By the assumption,
	$\mathcal{M}_k, n \models A$.
	\end{itemize}
\end{proof}

Conforming to \Smor 's argument, we prove the following theorem.

\begin{thm}\label{T1}
The class $\FIFD$ does not have the fixed-point property.
\end{thm}

\begin{proof}
Let $A$ be any $\mathcal{L}$-sentence containing only $P$. 
It suffices to show that there is $k \in \mathbb{N}$ such that $\mathcal{M}_k \not \models A \leftrightarrow \forall u \Box ( A \to P(u) )$. By Claim \ref{Sm}, the set $\{ n \in \mathbb{N} \mid \mathcal{M}_S, n \models A \}$ is either finite or co-finite. Then for some $k \in \mathbb{N}$, either
\begin{center}
$k$ is odd and $\mathcal{M}_S, k \models A$ \quad or \quad $k$ is even and $\mathcal{M}_S, k \not \models A$. 
\end{center}
By Lemma \ref{L2}, $\mathcal{M}_S, k \models A \Leftrightarrow \mathcal{M}_k, k \models A$. Therefore we have either
\begin{center}
$k$ is odd and $\mathcal{M}_k, k \models A$ \quad or \quad $k$ is even and $\mathcal{M}_k, k \not \models A$. 
\end{center}
By Lemma \ref{L3}, we conclude
$\mathcal{M}_k \not \models A \leftrightarrow \forall u \Box ( A \to P(u) )$. 
\end{proof}

\begin{cor}
The classes $\FH$ and $\FI$ do not have the fixed-point property.
\end{cor}

\section{The fixed-point theorem for $\QK + \Box^{n+1} \bot$ and the local fixed-point property for $\FH$}\label{QK}

In this section, we prove the fixed-point theorem for $\mathbf{QK} + \Box^{n+1} \bot$. 
Consequently, we show the class $\FH$ has the local fixed-point property.

\begin{thm}\label{T2}
Let $n \in \mathbb{N}$, and suppose that an $\mathcal{L}'$-formula $A(p)$ is modalized in $p$. Then there is an $\mathcal{L}$-formula $B$ such that $B$ contains only predicate symbols and free variables occurring in $A(p)$, and
	\begin{equation*}
	\QK \prove \Box^{n+1} \bot \to ( B \leftrightarrow A(B) ).
	\end{equation*}
Moreover, such a formula $B$ is effectively calculable from $A(p)$. 
\end{thm}

Before proving Theorem \ref{T2}, we give some definitions, and prove several lemmas. 

\begin{defi}\leavevmode \label{D1}
	\begin{enumerate}
	\item Let $A$ be an $\mathcal{L}'$-formula, and $B$ be a subformula of $A$. The \textit{depth} of an occurrence of $B$ in $A$ is the total number of subformulas $\Box C$ of $A$, containing the occurrence of $B$, not $B$ itself.
	\item For an $\mathcal{L}'$-formula $A$, $A^{\top (n)}$ denotes the formula obtained from $A$
	by replacing every occurrence of the form $\Box B$ of depth $n$ by $\top$.
	\item For an $\mathcal{L}'$-formula $A(p)$, $A(p) [ B_0, \ldots, B_n]$ denotes the formula obtained from $A(p)$ by substituting $B_i$ for all occurrences of $p$ of depth $i$ for each $i \leq n$, respectively.
	\end{enumerate}
\end{defi}

For instance, put $A(p) : \equiv \Box \left( p \to  \forall u ( Q(u) \to \Box p) \right)$. Then the depth of $A$ is $0$, and the depth of $\Box p$ is $1$.
By Definition \ref{D1}.2, 
\begin{gather*}
A^{\top(0)} \equiv \top, \quad A^{\top(1)} \equiv \Box \left( p \to \forall u \left( Q(u) \to \top \right) \right), \text{ and } A^{\top(2)} \equiv A.
\end{gather*}
The depth of the left $p$ is $1$, and the depth of the right $p$ is $2$. By Definition \ref{D1}.3, 
\begin{gather*}
A(p) [ B_0, B_1, B_2] \equiv \Box \left( B_1 \to \forall u \left( Q(u) \to \Box B_2 \right) \right).
\end{gather*}

The following lemma immediately follows from Definition \ref{D1}.

\begin{lem}\label{L15}
Let $m, n \in \mathbb{N}$ with $m \geq n$. Let $A(p)$ be any $\mathcal{L}'$-formula,  and $B_0, \ldots B_m$ be any $\mathcal{L}$-formulas.
Then the followings hold:
	\begin{enumerate}
	\item $A^{\top(n)}$ contains only occurrences of $p$ of depth $\leq n$.
	Thus $A^{\top(n)} (p) \left[ B_0, \ldots, B_n \right]$ is an $\mathcal{L}$-formula;
	\item ${\left( A^{\top(m)} \right)}^{\top(n)} \equiv A^{\top(n)}$;
	\item ${\left( A(p) \left[ B_0, \ldots, B_m\right]\right)}^{\top(n)} \equiv A^{\top(n)}(p) \left[ B_0, \ldots, B_n \right]$.
	\end{enumerate} 
\end{lem}

\begin{lem}\label{L4}
For any $n\in \mathbb{N}$ and $\mathcal{L}$-formula $A$,
	\begin{equation*}
	\QK \prove \Box^{n+1} \bot \to \left( A \leftrightarrow A^{\top (n)} \right).
	\end{equation*}
\end{lem}

\begin{proof}
By the induction on the construction of $A$, we show that for any $n \in \mathbb{N}$, $\QK \prove \Box^{n+1} \bot \to \left( A \leftrightarrow A^{\top (n)} \right)$.
	\begin{itemize}
	\item  If $A$ is an atomic formula, then for any $n \in \mathbb{N}$, $A^{\top(n)} \equiv A$.
	Clearly $\QK \prove A \leftrightarrow A^{\top(n)}$, and hence $\QK \prove \Box^{n+1} \bot \to \left( A \leftrightarrow A^{\top (n)} \right)$.
	\item The cases for $A \equiv \neg B$ and $A \equiv B \to C$, Lemma clearly follows from
	the definition of $A^{\top(n)}$ and the induction hypothesis.
	\item  Suppose that $A \equiv \forall u B$, and Lemma holds for $B$. 
	In this case for any $n \in \mathbb{N}$, $A^{\top(n)} \equiv \forall u \left( B^{\top(n)} \right)$. 
	By the induction hypothesis, $\QK \prove \Box^{n+1} \bot \to \left( B \leftrightarrow B^{\top(n)} \right)$ and hence
	$\QK \prove \Box^{n+1} \bot \to \left( \forall u B \leftrightarrow \forall u \left( B^{\top(n)} \right) \right)$.
	Therefore $\QK \prove \Box^{n+1} \bot \to \left( A \leftrightarrow A^{\top(n)} \right)$.
	\item Suppose that $A \equiv \Box B$ and Lemma holds for $B$.
	We distinguish the following two cases. 
		\begin{itemize}
		\item If $n=0$, then $A^{\top(0)} \equiv \top$. Since $\QK \prove \Box \bot \to (\Box B \leftrightarrow \top)$
		for any $\mathcal{L}$-formula $B$, 
		$\QK \prove \Box \bot \to \left( A \leftrightarrow A^{\top(0)} \right)$.

		\item Suppose that $n >0$. By the inductive hypothesis for $B$, 
		$\QK \prove \Box^{n} \bot \to \left( B \leftrightarrow B^{\top(n-1)} \right)$.
		By the derivation of $\QK$, we have
		$\QK \prove \Box^{n+1} \bot \to \left( \Box B \leftrightarrow \Box \left(B^{\top(n-1)} \right) \right)$.
		Note that each occurrence of $\Box C$ in $\Box B$ of depth $\geq n$ is the one in $B$
		of depth $\geq n-1$. Therefore
		$A^{\top(n)} \equiv \left( \Box B \right)^{\top(n)} \equiv \Box \left(B^{\top(n-1)} \right)$.
		Thus, $\QK \prove \Box^{n+1} \bot \to \left( A \leftrightarrow A^{\top(n)}\right)$.
		\end{itemize}
	\end{itemize}
\end{proof}

\begin{lem}\label{L5}
Suppose that $A(p)$ is an $\mathcal{L}'$-formula containing only occurrences of $p$ of depth $\leq n$, and $\mathcal{L}$-formulas $C_0, \ldots, C_n$ and $D_0, \ldots, D_n$ contain no free variables which are bounded in $A(p)$. Then
	\begin{align*}
	\QK & \prove \Box^{n+1} \bot \land \bigwedge_{i \leq n} \Box^{n-i} \left( \Box^{i+1} \bot \to \left( C_i \leftrightarrow D_i \right) \right) \\
	& \hspace{30mm} \to \left( A(p) \left[ C_n, \ldots, C_0 \right]  \leftrightarrow A(p) \left[ D_n, \ldots, D_0 \right] \right).
	\end{align*}
\end{lem}

\begin{proof}
Induction on the construction of $A(p)$.
	\begin{itemize}
	\item Assume $A(p) \equiv p$. Then for any $n \in \mathbb{N}$, the depth of each occurrence of $p$ is
	$\leq n$, and $A(p)$ contains no free variables. 
	For any $\mathcal{L}$-formula $C_0, \ldots, C_n$ and $D_0, \ldots, D_n$,
	$\QK \prove \left( C_n \leftrightarrow D_n \right) \leftrightarrow  \left( C_n \leftrightarrow D_n \right)$,
	and hence
		\begin{equation*}
		\QK \prove \Box^{n+1} \bot
		\land \left( \Box^{n+1} \bot \to \left( C_n \leftrightarrow D_n \right) \right)
		\to \left( C_n \leftrightarrow D_n \right).
		\end{equation*}
	Adding the assumptions, we obtain
		\begin{equation*}
		\QK \prove \Box^{n+1} \bot \land
		\bigwedge_{i \leq n} \Box^{n-i} \left( \Box^{i+1} \bot \to \left( C_i \leftrightarrow D_i \right) \right)
		\to ( C_n \leftrightarrow D_n ).
		\end{equation*}
	Since $A(p) [ C_n, \ldots, C_0] \equiv C_n$ and $A(p)[ D_n, \ldots, D_0 ] \equiv D_n$,
	Lemma holds for $A(p)$.
	\item Suppose that $A(p)$ is one of the form $\neg B(p)$, $B(p) \to C(p)$ or $\forall u B(p)$.
	If $A(p)$ contains only the occurrences of $p$ of depth $\leq n$, then so does $B(p)$ and $C(p)$.
	Moreover, for any $\mathcal{L}$-formula $F$, if all free variables occurring in $F$ are not bounded in $A(p)$,
	then they are not bounded in $B(p)$ and $C(p)$, too.
	By the induction hypothesis and the derivation of predicate logic, Lemma holds for $A(p)$.
	\item Assume $A(p) \equiv \Box B(p)$. If $A(p)$ contains only the occurrences of $p$ of depth $\leq n$,
	$B(p)$ contains only the occurrence of $p$ of depth $\leq n-1$.
	Let $C_0, \ldots, C_n$ and $D_0, \ldots, D_n$ be $\mathcal{L}$-formulas satisfying the assumption of Lemma.
	Every free variables occurring freely in $C_i$ or $D_i$ occur freely in $B(p)$. By the induction hypothesis,
		\begin{align*}
		\QK & \prove \Box^n \bot 
		\land \bigwedge_{i \leq n-1} \Box^{n-1-i} \left( \Box^{i+1} \bot
		\to \left( C_i \leftrightarrow D_i \right) \right) \\
		& \hspace{30mm} \to \left( B(p) [ C_{n-1}, \ldots, C_0 ]
		\leftrightarrow B(p) [ D_{n-1}, \ldots, D_0 ] \right).
		\end{align*}
	By the derivation of $\QK$,
		\begin{align*}
		\QK & \prove \Box^{n+1} \bot 
		\land \bigwedge_{i \leq n-1} \Box^{n-i} \left( \Box^{i+1} \bot \to ( C_i \leftrightarrow D_i ) \right) \\
		& \hspace{18mm} \to \left( \Box \left( B(p) [ C_{n-1}, \ldots, C_0 ] \right)
		\leftrightarrow \Box \left( B(p) [ D_{n-1}, \ldots, D_0 ] \right) \right).
		\end{align*}
	Since $A(p)$ does not contain the occurrence of $p$ of depth $0$,
		\begin{align*}
		\Box \left( B(p) [ C_{n-1}, \ldots, C_0 ] \right) & \equiv A(p) [ C_n, \ldots, C_0 ], \text{ and } \\
		\Box \left( B(p) [ D_{n-1}, \ldots, D_0 ] \right) & \equiv A(p) [ D_n, \ldots, D_0 ].
		\end{align*}
	Therefore
		\begin{align*}
		\QK & \prove \Box^{n+1} \bot \quad \land \bigwedge_{i \leq n-1} \Box^{n-i} \left( \Box^{i+1} \bot \to ( C_i \leftrightarrow D_i ) \right) \\
		& \hspace{40mm} \to \left( A(p) [ C_n, \ldots, C_0 ] \leftrightarrow A(p) [ D_n, \ldots, D_0 ] \right).
		\end{align*}
	Adding the assumptions, we obtain
		\begin{align*}
		\QK & \prove \Box^{n+1} \bot 
		\land \bigwedge_{i \leq n} \Box^{n-i} \left( \Box^{i+1} \bot \to \left( C_i \leftrightarrow D_i \right) \right)
		\\
		& \hspace{40mm} \to \left( A(p) [ C_n, \ldots, C_0 ]  \leftrightarrow A(p) [ D_n, \ldots, D_0 ] \right).
		\end{align*}
	\end{itemize}
\end{proof}

In the remainder of this section, we fix an $\mathcal{L}'$-formula $A(p)$ which is modalized in $p$, i.e., $A(p)$ contains no occurrences of $p$ of depth $0$. By replacing variables appropriately, we assume that every free variable occurring in $A(p)$ does not occur in $A(p)$ as a bound variable. We define the sequence $\{ A_n \}_{n < \omega}$ of $\mathcal{L}$-formulas recursively as follows:
\begin{enumerate}
\item $A_0 :\equiv A^{\top(0)}(p) \left[ \top \right] \left(\equiv A^{\top(0)}(p) \right)$;
\item $A_{n+1} :\equiv A^{\top(n+1)}(p)\left[\top, A_n, \ldots, A_0\right]$.
\end{enumerate}

By the definition and Lemma \ref{L15}.1, every $A_n$ is an $\mathcal{L}$-formula and contains only predicate symbols and free variables occurring in $A(p)$.

\begin{lem}\label{L6}
For any $m, n \in \mathbb{N}$, if $m \geq n$, then $\QK \prove \Box^{n+1} \bot \to (A_m \leftrightarrow A_n)$.
\end{lem}

\begin{proof}
Induction on $n$.
	\begin{itemize}
	\item Assume $n=0$, and take $m \geq 0$ arbitrarily. Then
		\begin{align*}
		A_m^{\top(0)} & \equiv {\left(A^{\top(m)}(p)[ \top, A_{m-1}, \ldots, A_0 ] \right)}^{\top(0)}, \\
		 & \equiv {\left( A^{\top(m)}\right)}^{\top(0)} (p) [ \top ], \tag{by Lemma \ref{L15}.3}\\
		 & \equiv A^{\top(0)} (p) [ \top ], \tag{by Lemma \ref{L15}.2}\\
		 & \equiv A_0.
		\end{align*}
	By Lemma \ref{L4}, $\QK \prove \Box \bot \to \left( A_m \leftrightarrow A_m^{\top(0)} \right)$.
	Thus we have $\QK \prove \Box \bot \to ( A_m \leftrightarrow A_0 )$.
	\item Suppose that Lemma holds for $\leq n$. Take $m+1 \geq n+1$ arbitrarily.
	Then by the induction hypothesis,
		\begin{equation*}
		\QK \prove \bigwedge_{i < n+1} \Box^{i+1} \bot \to ( A_{i+(m-n)} \leftrightarrow A_i ),
		\end{equation*}
	and hence
		\begin{align*}
		\QK \prove \bigwedge_{i < n+1} \Box^{n+1-i} ( \Box^{i+1} \bot \to (A_{i+(m-n)} \leftrightarrow A_i)).
		\end{align*}
	Note that $\QK \prove \Box^0 ( \Box^{n+2} \bot \to (\top \leftrightarrow \top))$,\footnote{Here $\Box^0 A \equiv A$.} and
	$A^{\top(n+1)}(p)$ contains no free variables which is bounded in each $A_i$. From them and
	by Lemma \ref{L5}, we obtain
		\begin{align}
		\QK &\prove \Box^{n+2} \bot \nonumber\\
		& \to \left( A^{\top(n+1)}(p) [ \top, A_m, \ldots, A_{m-n} ] 
		\leftrightarrow A^{\top(n+1)}(p) [ \top, A_n, \ldots, A_0] \right). \label{E9}
		\end{align}
	On the other hand, by Lemma \ref{L4}, $\QK \prove \Box^{n+2}\bot \to \left( A_{m+1} \leftrightarrow A_{m+1}^{\top(n+1)} \right)$.	
	Recall that
		\begin{align*}
		A_{m+1}^{\top(n+1)} & \equiv \left( A^{\top(m+1)}(p) [ \top, A_m, \ldots, A_0 ] \right)^{\top(n+1)}, \\
		& \equiv {\left( A^{\top(m+1)} \right)}^{\top(n+1)} (p) [ \top, A_m, \ldots, A_{m-n} ], \tag{by Lemma \ref{L15}.3} \\
		& \equiv A^{\top(n+1)} (p) [ \top, A_m, \ldots, A_{m-n} ], & \tag{by Lemma \ref{L15}.2}
		\end{align*}
	Thus
		\begin{align}
		\QK \prove \Box^{n+2}\bot \to \left( A_{m+1} \leftrightarrow A^{\top(n+1)} (p) [ \top, A_m, \ldots, A_{m-n} ] \right).
		\label{E10}
		\end{align}
	From (\ref{E9}) and (\ref{E10}), we conclude $\QK \prove \Box^{n+2} \bot \to ( A_{m+1} \leftrightarrow A_{n+1})$.
	\end{itemize}
\end{proof}

Let $B(p)$ be an $\mathcal{L}'$-formula. For $n \in \mathbb{N}$, we define
	\begin{equation*}
	B^n :\equiv B^{\top(n)}(p) [ A_n, \ldots, A_0 ].
	\end{equation*}
By Lemma \ref{L15}.1, the formula $B^n$ is an $\mathcal{L}$-formula. 
Since $A(p)$ is modalized in $p$, we obtain 
	\begin{align*}
	A^n & \equiv A^{\top(n)}(p) [A_n, A_{n-1}, \ldots, A_0], \\
	 & \equiv A^{\top(n)}(p) [ \top, A_{n-1}, \ldots, A_0 ], \\
	 & \equiv A_n.
	\end{align*}

\begin{lem}\label{L7}
For any $\mathcal{L}'$-formula $B(p)$ and $m,n \in \mathbb{N}$, if $m \geq n$, then
	\begin{equation*}
	\QK \prove \Box^{n+1} \bot \to \left( B^n \leftrightarrow B(A_m) \right).
	\end{equation*}
\end{lem}

\begin{proof}
Induction on the construction of $B(p)$. Assume $m \geq n$.
	\begin{itemize}
	\item Assume $B(p) \equiv p$. In this case, $B^n \equiv B^{\top(n)}(p)[ A_n, \ldots, A_0 ] \equiv A_n$, and
	$B(A_m) \equiv A_m$. By Lemma \ref{L6}, $\QK \prove \Box^{n+1} \bot \to ( A_m \leftrightarrow A_n )$.
	Therefore $\QK \prove \Box^{n+1} \bot \to ( B^n \leftrightarrow B(A_m) )$.
	\item The cases for $B(p) \equiv \neg C(p)$ and $B(p) \equiv C(p) \to D(p)$ are clear.
	\item Assume $B(p) \equiv \forall u C(p)$ and Lemma holds for $C(p)$. By the induction hypothesis, 
	$\QK \prove \Box^{n+1} \bot \to ( C^n \leftrightarrow C(A_m))$. Recall that
	$\forall u (C^n) \equiv (\forall u C)^n$. By the generalization, we have
	$\QK \prove \Box^{n+1} \bot \to ( (\forall u C)^n \leftrightarrow \forall u C(A_m) )$, 
	i.e., $\QK \prove \Box^{n+1} \bot \to ( B^n \leftrightarrow B(A_m))$.
	\item Assume $B(p) \equiv \Box C(p)$ and Lemma holds for $C(p)$.
	We distinguish the following two cases. 
		\begin{itemize}
		\item If $n=0$, then we have $B^0 \equiv {(\Box C)}^0 \equiv {(\Box C)}^{\top(0)}(p) [ A_0] \equiv \top$.
		Since $\QK \prove \Box \bot \to \Box C(A_m)$, we obtain
		$\QK \prove \Box \bot \to ( B^0 \leftrightarrow B(A_m))$.
		\item Suppose that $n >0$. Take $m \geq n$ arbitrarily. Then $m > n-1$.
 		By the induction hypothesis for $C(p)$, $m$ and $n-1$,
		$\QK \prove \Box^n \bot \to \left( C^{n-1} \leftrightarrow C(A_m) \right)$. By the derivation of $\QK$, we have
		$\QK \prove \Box^{n+1} \bot \to \left( \Box (C^{n-1}) \leftrightarrow \Box C(A_m) \right)$. Since $B(p)$ contains 
		no occurrences of $p$ of depth $0$, we obtain
			\begin{align*}
			\Box (C^{n-1}) & \equiv \Box \left( C^{\top(n-1)}(p) [ A_{n-1},\ldots, A_0 ] \right) \\
			& \equiv {( \Box C)}^{\top(n)}(p)[ A_n, A_{n-1}, \ldots, A_0 ] \\
			& \equiv B^{n}.
			\end{align*}
		Thus, $\QK \prove \Box^{n+1}\bot \to \left( B^n \leftrightarrow B(A_m) \right)$.
		\end{itemize}
	\end{itemize}
\end{proof}

Here we are ready to prove Theorem \ref{T2}. 

\begin{proof}[Proof of Theorem \ref{T2}]
Let $A(p)$ be the fixed $\mathcal{L}'$-formula which is modalized in $p$, and it suffices to show that $A_n$ is a fixed-point of $A(p)$ in $\QK + \Box^{n+1}\bot$. By Lemma \ref{L7},  we obtain $\QK \prove \Box^{n+1} \bot \to \left(A^n \leftrightarrow A(A_n) \right)$.
Since $A^n \equiv A_n$, $\QK \prove \Box^{n+1} \bot \to ( A_n \leftrightarrow A(A_n))$. The formula $A_n$ contains only predicate symbols and free variables occurring in $A$. Thus, $A_n$ is a fixed-point of $A(p)$ in $\QK + \Box^{n+1}\bot$.
\end{proof}

\begin{rem}
In \cite{Sac99}, Sacchetti proved the fixed-point theorem for propositional modal logics $\mathbf{K} + \Box^{n+1} \bot$ without giving an algorithm for calculating fixed-points in these logics. 
Our proof of Theorem \ref{T2} provides such an algorithm even for the logics $\mathbf{K} + \Box^{n+1} \bot$. 
\end{rem}

\begin{cor}\label{C1}
The classes $\FH$, $\FI$ and $\FIFD$ have the local fixed-point properties.
\end{cor}

\begin{proof}
It is sufficient to prove only the case for $\FH$. Let $\mathcal{F} = \langle W, \prec, {\{ D_w\}}_{w \in W} \rangle$ be a Kripke frame in the class $\FH$. Put $h(\mathcal{F})=n$. Then for any $w \in W$, $h(w) \leq n$, i.e., $\mathcal{F} \models \Box^{n+1} \bot$. Let $A(p)$ be any $\mathcal{L}'$-formula which is modalized in $p$. From Theorem \ref{T2}, we have $\QK \prove \Box^{n+1} \bot \to \left( A_n \leftrightarrow A(A_n) \right)$. Recall that $\QK \subseteq \QGL \subseteq \mathbf{MQ}(\FH)$. Thus we have $\mathcal{F} \models \Box^{n+1} \bot \to \left( A_n \leftrightarrow A(A_n) \right)$. From this and $\mathcal{F} \models \Box^{n+1} \bot$, we conclude $\mathcal{F} \models A_n \leftrightarrow A(A_n)$. The formula $A_n$ is indeed a local fixed-point of $A(p)$ in $\mathcal{F}$. 
\end{proof}

In Section 3, we proved that the class $\FIFD$ does not have the fixed-point property (Theorem \ref{T1}). Corollary \ref{C1} shows that $\mathbf{MQ}(\FIFD)$ is consistent with the fixed-point property, that is, there is a consistent extension of $\mathbf{MQ}(\FIFD)$ for which the fixed-point theorem holds.

In Section \ref{Sec2-1}, we mentioned that $\mathbf{MQ}(\BL)$ equals to $\mathbf{MQ}(\FH)$, and thus the classes $\BL$ and $\FH$ cannot be distinguished by the validity of formulas. On the other hand, $\BL$ does not have the local fixed-point property (Corollary \ref{C2}), and $\FH$ has the one (Corollary \ref{C1}). Hence we can capture some logical difference between the classes $\BL$ and $\FH$ through the local fixed-point property.

\section{Failure of the Craig interpolation property for $\NQGL$}\label{Sec5}

In this section, we prove that the logic $\NQGL$ does not enjoy the Craig interpolation property.

\begin{defi}
We say a logic $\mathbf{L}$ enjoys the Craig interpolation property if for any sentences $A$ and $B$, if $\mathbf{L}$ proves $A \to B$, then there exists a sentence $C$ containing only predicate symbols occurring in both $A$ and $B$ such that $\mathbf{L}$ proves $A \to C$ and $C \to B$. 
\end{defi}

\begin{thm}\label{T3}
The system $\NQGL$ does not have the Craig interpolation property.
\end{thm}

Before proving Theorem \ref{T3}, we prepare several lemmas. 

\begin{lem}\label{L13}
Suppose that $A(p)$ is an $\mathcal{L}'$-formula not containing the unary predicate $P$,
and not containing occurrences of $u$ and $v$ as bound variables.
If $\NQGL \prove \forall u A\left(P(u)\right)$, then for any $\mathcal{L}'$-formula $B(v)$,
$\NQGL \prove \forall v A\left(B(v)\right)$.
\end{lem}

\begin{proof}
Suppose that for some $B(v)$, $\NQGL \nvdash \forall v A \left( B(v) \right)$. By Theorem \ref{Tana}, there exists a Kripke model $\mathcal{M} = \langle \mathcal{F}, \Vdash \rangle = \langle W, \prec, {\{ D_w \}}_{w \in W}, \Vdash \rangle$ such that $\mathcal{F} \in \FH$, and for some $w \in W$ and $c \in D_w$, $\mathcal{M}, w \not\models A \left( B(c) \right)$. We may assume $w$ is the root of $\mathcal{F}$. Then for every $x \in W$, $c \in D_x$. We define an interpretation $\Vdash^\ast$ of $\mathcal{F}$ as follows:
	\begin{itemize}
	\item For any predicate symbol $Q$ other than $P$, 
	${\Vdash^\ast \langle w, Q \rangle} = {\Vdash \langle w, Q \rangle}$ for every $w \in W$;
	\item For every $x \in W$ and $a \in D_x$,
	$x \Vdash^\ast P(a) :\Leftrightarrow x \Vdash B(c)$.
	\end{itemize} 
Let $\mathcal{M}^\ast := \langle \mathcal{F}, \Vdash^\ast \rangle$. We claim that for any $\mathcal{L}'$-formula $C(p)$, $x \in W$ and $a \in D_x$, $\mathcal{M}, x \models C \left( B(c) \right) \iff \mathcal{M}^\ast, x \models C \left( P(a) \right)$. We prove the claim by induction on the construction of $C(p)$.
	\begin{itemize}
	\item If $C(p)$ contains no occurrences of $p$, then the claim trivially holds.
	\item Assume $C(p) \equiv p$. Then $ C \left( B(c) \right) \equiv B(c)$ and $C \left( P(a) \right) \equiv P(a)$.
	By the definition of $\Vdash^\ast$, we have
	$\mathcal{M}, x \models C \left( B(c) \right) \iff \mathcal{M}^\ast, x \models C \left( P(a) \right)$.
	\item The cases $C(p) \equiv \neg D(p)$ and $C(p) \equiv D(p) \to E(p)$ are clear by the induction hypothesis.
	\item Assume $C(p) \equiv \forall v D(p)$. Then
		\begin{align*}
		\mathcal{M}, x \models \forall v D \left( B(c) \right) &\iff \mathcal{M}, x \models D \left( B(c) \right)[v / b]
		\text{ for all } b \in D_x, \\
		&\iff  \mathcal{M}^\ast, x \models D \left( P(a) \right)[v / b]
		\text{ for all } b \in D_x,  \tag{I.H.} \\
		&\iff \mathcal{M}^\ast, x \models \forall v D \left( P(a) \right).
		\end{align*}
	\item Assume $C(p) \equiv \Box D(p)$. Then
		\begin{align*}
		\mathcal{M}, x \models \Box D\left(B(c)\right) &\iff \mathcal{M}, y \models D\left(B(c)\right) \text{ for any } y \succ x, \\
		&\iff \mathcal{M}^\ast, y \models D\left( P(a) \right) \text{ for any } y \succ x,  \tag{I.H.} \\
		&\iff \mathcal{M}^\ast, x \models \Box D \left( P(a) \right).
		\end{align*}
	\end{itemize}
The proof of the claim is completed. From $\mathcal{M}, w \not\models A \left( B(c) \right)$ and by the claim,
$\mathcal{M}^\ast, w \not\models A \left( P(a) \right)$, and hence $\mathcal{M}^\ast, w \not\models \forall u A \left( P(u) \right)$.
By Theorem \ref{Tana}, $\NQGL \nvdash \forall u A \left( P(u) \right)$.
\end{proof}

We prove the following uniqueness lemma of fixed-points in $\NQGL$.

\begin{lem}[Uniqueness of fixed-points in $\NQGL$] \label{L14}
Let $A(p)$ be any $\mathcal{L}'$-formula which is modalized in $p$. 
Let $F_0$ and $F_1$ be any $\mathcal{L}$-formulas which contain no bounded variables occurring freely in $A(p)$. 
Then 
	\[
	\NQGL \prove \boxdot \left( A \left(F_0 \right) \leftrightarrow F_0 \right)
	\land \boxdot \left( A \left(F_1\right) \leftrightarrow F_1 \right) \to \left( F_0 \leftrightarrow F_1 \right).
	\]
\end{lem}

\begin{proof}
We claim that, for any $n \in \mathbb{N}$, $\mathcal{L}'$-formula $A(p)$ which is modalized in $p$, and $\mathcal{L}$-formula $F$ which contains no bounded variables occurring freely in $A(p)$,
	\begin{equation*}
	\QGL \prove \Box^{n+1} \bot \to \left( \boxdot \left( A(F) \leftrightarrow F \right) \to (F \leftrightarrow A_n) \right), 
	\end{equation*}
where $A_n$ is the $\mathcal{L}$-formula defined in Section \ref{QK}. 
By Lemma \ref{L16},
$\QKf \prove \Box \left( F \leftrightarrow A_n \right) \to \left( A(F) \leftrightarrow A(A_n) \right)$.
By Theorem \ref{T2}, $\QK \prove \Box^{n+1} \bot \to \left(A(A_n) \leftrightarrow A_n \right)$. Thus
$\QKf \prove \Box^{n+1} \bot \to \left( \Box \left( F \leftrightarrow A_n \right) \to \left( A(F) \leftrightarrow A_n \right) \right)$.
Then 
	\begin{align}
	\QKf &\prove \Box^{n+1} \bot \land \left( A(F) \leftrightarrow F \right) 
	\to \left( \Box \left( F \leftrightarrow A_n \right) \to \left( F \leftrightarrow A_n \right) \right), \label{E4}\\
	\QKf &\prove \Box^{n+2} \bot \land \Box \left( A(F) \leftrightarrow F \right)
	\to \Box \left( \Box \left( F \leftrightarrow A_n \right) \to \left( F \leftrightarrow A_n \right) \right), \nonumber\\
	\QGL &\prove \Box^{n+2} \bot \land \Box \left( A(F) \leftrightarrow F \right)
	\to \Box \left( F \leftrightarrow A_n \right). & \tag{ by \textbf{\Lob}} \nonumber
	\end{align}
Since $\QKf \prove \Box^{n+1}\bot \to \Box^{n+2}\bot$, we obtain
	\begin{equation*}
	\QGL \prove \Box^{n+1} \bot \land \Box \left( A(F) \leftrightarrow F \right) \to \Box \left( F \leftrightarrow A_n \right).
	\end{equation*}
From this and (\ref{E4}),
$\QGL \prove \Box^{n+1} \bot \to \left( \boxdot \left( A(F) \leftrightarrow F \right) \to (F \leftrightarrow A_n) \right)$.
The proof of the claim is completed. 

Let $A(p)$, $F_0$ and $F_1$ be formulas as in the statement of Lemma. 
By the claim, for any $n \in \mathbb{N}$,
	\begin{align*}
	\QGL &\prove \Box^{n+1}\bot \to \left( \boxdot \left( A \left(F_0 \right) \leftrightarrow F_0 \right)
	\to \left(F_0 \leftrightarrow A_n \right) \right), \text{ and }\\
	\QGL &\prove \Box^{n+1}\bot \to \left( \boxdot \left( A \left(F_1 \right) \leftrightarrow F_1 \right)
	\to \left(F_1 \leftrightarrow A_n \right) \right).
	\end{align*}
Therefore
\begin{align*}
	\QGL \prove \Box^{n+1}\bot \to \left(\boxdot \left( A \left(F_0 \right) \leftrightarrow F_0 \right) \land \boxdot \left( A \left(F_1 \right) \leftrightarrow F_1 \right) \to \left(F_0 \leftrightarrow F_1 \right) \right). 
\end{align*}
Applying the rule $\mathbf{BL}$ of $\NQGL$, we conclude
	\begin{align*}
	\NQGL \prove \boxdot \left( A \left(F_0 \right) \leftrightarrow F_0 \right) \land \boxdot \left( A \left(F_1 \right) \leftrightarrow F_1 \right) \to \left(F_0 \leftrightarrow F_1 \right). 
	\end{align*}
\end{proof}

\begin{proof}[Proof of Theorem \ref{T3}]
Let $A(p) \equiv \forall u \Box \left( p \to P(u) \right)$. By Lemma \ref{L14}, for any unary predicate symbols $Q$ and $R$ other than $P$, and any variables $v_0$ and $v_1$,
	\begin{align*}
	\NQGL & \prove \boxdot \left( A \left( Q(v_0) \right) \leftrightarrow Q(v_0) \right)
	\land \boxdot \left( A \left( R(v_1) \right) \leftrightarrow R(v_1) \right)
	\to \left( Q(v_0) \leftrightarrow R(v_1)\right), \\
	\NQGL & \prove \forall v_0 \forall v_1 \left( 
	\boxdot \left( A \left( Q(v_0) \right) \leftrightarrow Q(v_0) \right)
	\land \boxdot \left( A \left( R(v_1) \right) \leftrightarrow R(v_1) \right) \right. \\
	& \hspace{76mm} \left. \to \left( Q(v_0) \leftrightarrow R(v_1) \right)
	\right),
	\end{align*}
and hence
	\begin{align}
	\NQGL &\prove \exists v_0 \left( 
	\boxdot \left( A \left( Q(v_0) \right) \leftrightarrow Q(v_0) \right)
	\land Q(v_0) \right) \nonumber \\
	&\hspace{27mm}\to \forall v_1 \left(
	\boxdot \left( A \left( R(v_1) \right) \leftrightarrow R(v_1) \right) 
	\to R(v_1) \right). \label{E5}
	\end{align}

We show that the implication (\ref{E5}) has no Craig interpolants. Suppose, for the contradiction, that (\ref{E5}) has a Craig interpolant $G$, then $G$ is an $\mathcal{L}$-sentence containing only the predicate symbol $P$ such that
	\begin{align*}
	\NQGL & \prove \exists v_0 \left( \boxdot \left( A\left( Q(v_0) \right) \leftrightarrow Q(v_0) \right) \land Q(v_0) \right)
	\to G, \text{ and}\\
	\NQGL & \prove G \to
	\forall v_1 \left( \boxdot \left( A \left( R(v_1) \right) \leftrightarrow R(v_1) \right) \to R(v_1) \right).
	\end{align*}
Hence
	\begin{align}
	\NQGL & \prove \forall v_0 \left( \boxdot \left( A\left( Q(v_0) \right) \leftrightarrow Q(v_0) \right) 
	\to \left( Q(v_0) \to G\right) \right), \text{ and} \label{E6}\\
	\NQGL & \prove \forall v_1 \left( \boxdot \left( A \left( R(v_1) \right) \leftrightarrow R(v_1) \right)
	\to \left( G \to R(v_1)\right) \right). \label{E7}
	\end{align}
We may assume $G$ does not contain $v_0$ and $v_1$. By Lemma \ref{L13}, substituting $Q(v_0)$ for $R(v_1)$ in (\ref{E7}), we have 
$\NQGL \prove \forall v_0 \left( \boxdot \left( A \left( Q(v_0) \right) \leftrightarrow Q(v_0) \right)
\to \left( G \to Q(v_0)\right) \right)$. From this and (\ref{E6}),
	\begin{equation*}
	\NQGL \prove \forall v_0 \left( \boxdot \left( A \left( Q(v_0) \right) \leftrightarrow Q(v_0) \right)
	\to \left( Q(v_0) \leftrightarrow G \right) \right). 
	\end{equation*}
By Lemma \ref{L13}, substituting $A(G)$ for $Q(v_0)$, we have
	\begin{equation}
	\NQGL \prove \boxdot \left( A \left( A(G) \right) \leftrightarrow A(G) \right) \to \left( A(G) \leftrightarrow G \right). \label{E8}
	\end{equation}
By the derivation of $\QKf$, we get $\NQGL \prove \Box \left( A \left( A(G) \right) \leftrightarrow A(G) \right) 
\to \Box \left( A(G) \leftrightarrow G \right)$. By Lemma \ref{L16}, $\QKf \prove \Box (A(G) \leftrightarrow G) \to (A(A(G)) \leftrightarrow A(G))$. Thus
	\begin{equation*}
	\NQGL \prove \Box \left( A \left( A(G) \right) \leftrightarrow A(G) \right) 
	\to \left( A\left( A(G) \right) \leftrightarrow A(G) \right).
	\end{equation*}
Since the \Lob\ rule is admissible in $\NQGL$, we obtain $\NQGL \prove A \left( A(G) \right) \leftrightarrow A(G)$, and hence
$\NQGL \prove \boxdot \left( A \left( A(G) \right) \leftrightarrow A(G) \right)$. From this and (\ref{E8}),
	\begin{equation*}
	\NQGL\prove A(G) \leftrightarrow G.
	\end{equation*}
This means that $G$ would be a fixed-point of $A(p)$ in $\NQGL$. However, Corollary 2.10.2 says that $A(p)$ has no fixed-points in $\NQGL$, contradiction.
\end{proof}

\section{Formulas having a fixed-point in $\QGL$}

In this section, we investigate a sufficient condition for formulas to have a fixed-points in $\QGL$. 
We introduce the notion of $\Sigma$-formulas, and then we prove that if $A(p)$ is a Boolean combination of $\Sigma$ formulas and formulas without $p$, then $A(p)$ has a fixed-point in $\QGL$. 

Let $\mathcal{L}''$ be the language $\mathcal{L}$ together with Boolean connectives $\lor, \land$, the existential quantifier $\exists$, and \textit{countably infinite} propositional variables $p, q, \ldots$. 
We assume that an $\mathcal{L}''$-formula $A(p)$ may contain propositional variables other than $p$. 
Let $\QGL''$ be the natural extension of the system $\QGL$ to the language $\mathcal{L}''$. 
It is easy to show that if an $\mathcal{L}''$-formula $A$ is proved in $\QGL''$, then the $\mathcal{L}$-formula obtained by substituting $\top$ for all propositional variables appearing in $A$ is proved in $\QGL$. 
This shows that the system $\QGL''$ is a conservative extension of $\QGL$. 
Thus in this section, we write simply $\QGL$ instead of $\QGL''$. 
Also it is easy to see that the substitution lemma (Lemma \ref{L16}) is extended to the language $\mathcal{L}''$.

\begin{defi}[$\Sigma$-formulas]\label{D2}
\textit{$\Sigma$-formulas} are defined inductively as follows: 
	\begin{itemize}
	\item An $\mathcal{L}''$-formula of the form $\Box B$ is a $\Sigma$-formula;
	\item If $B$ and $C$ are $\Sigma$-formulas, then $B \lor C$, $B \land C$ and $\exists u B$ are $\Sigma$-formulas.
	\end{itemize}
\end{defi}

If $A(p)$ is a $\Sigma$-formula, then $A(p)$ contains no occurrences of $p$ of depth $0$, and for any $\mathcal{L}''$-formula $B$, the formula $A(B)$ is also a $\Sigma$-formula.

\begin{thm}\label{T4}
If $A(p)$ is a Boolean combination of $\Sigma$-formulas and $\mathcal{L}''$-formulas containing no occurrences of $p$, then there exist an $\mathcal{L}''$-formula $F$ such that $F$ contains only predicate symbols, propositional variables, free variables occurring in $A(p)$, not containing $p$, and such that $\QGL \prove F \leftrightarrow A(F)$.
\end{thm}

Before proving the theorem, we give a definition and prove some lemmas.

	\begin{defi}[Self-provers]
	An $\mathcal{L}''$-formula $A$ is said to be a \textit{self-prover} if $\QGL \prove A \to \Box A$.
	\end{defi}

	\begin{lem}\label{L8}
	The Boolean constant $\top$ and $\mathcal{L}''$-formulas of the form $\Box A$ are self-provers.
	Moreover, the set of all self-provers is closed under $\land, \lor, \exists$. Consequently,
	every $\Sigma$-formula is a self-prover.
	\end{lem}
	
	\begin{proof}
	Since $\QGL \prove \top \to \Box \top$ and $\QGL \prove \Box A \to \Box \Box A$, 
	$\top$ and $\Box A$ are self-provers. Suppose that $A$ and $B$ are self-provers.
		\begin{itemize}
		\item Since $A$ and $B$ are self-provers, $\QGL \prove A \land B \to \Box A \land \Box B$. 
		On the other hand, $\QGL \prove \Box A \land \Box B \to \Box ( A \land B)$. Thus we have
		$\QGL \prove A \land B \to \Box ( A \land B )$, and hence $A \land B$ is a self-prover.
		\item Since $\QGL \prove A \to A \lor B$, we have $\QGL \prove \Box A \to \Box ( A \lor B)$.
		Since $A$ is a self-prover, we get $\QGL \prove A \to \Box (A \lor B)$. By a similar argument,
		$\QGL \prove B \to \Box ( A \lor B )$. Thus, $\QGL \prove A \lor B \to \Box ( A \lor B )$,
		and hence $A \lor B$ is a self-prover. 
		\item Since $\QGL \prove A \to \Box A$, we have $\QGL \prove \exists u A \to \exists u \Box A$.
		On the other hand, from $\QGL \prove A \to \exists u A$, we have $\QGL \prove \Box A \to \Box \exists u A$,
		and hence $\QGL \prove \exists u \Box A \to \Box \exists u A$.
		Thus, $\QGL \prove \exists u A \to \Box \exists u A$, and hence $\exists u A$ is a self-prover.
		\end{itemize}
	\end{proof}

	\begin{lem}\label{L10}
	Let $A$ and $B$ be self-provers. If $\QGL \prove \Box A \to (A \leftrightarrow B)$,
	then $\QGL \prove A \leftrightarrow B$.
	\end{lem}

	\begin{proof}
	Since $A$ is a self-prover, $\QGL \prove A \to \Box A$. From this and the assumption,
	$\QGL \prove A \to ( A \leftrightarrow B)$, and hence $\QGL \prove A \to B$.
	On the other hand, by the assumption, $\QGL \prove B \to (\Box A \to A)$, and hence
	$\QGL \prove \Box B \to \Box (\Box A \to A)$. Applying the axiom of $\QGL$, we get 
	$\QGL \prove \Box B \to \Box A$. Since $B$ is a self-prover, $\QGL \prove B \to \Box A$. From this
	and the assumption, $\QGL \prove B \to (A \leftrightarrow B)$, and hence $\QGL \prove B \to A$. Thus $\QGL \prove A \leftrightarrow B$.
	\end{proof}

We assume that, by replacing variables appropriately, for any formula $A$, the set of free variables of $A$ and the set of bound variables of $A$ are disjoint. ($\dag$)

	\begin{lem}\label{L11}
	For any $\Sigma$-formula $S(p)$, there is an $\mathcal{L}''$-formula $F$ containing
	only predicate symbols, propositional variables and free variables occurring in $S$,
	not containing $p$, and such that $\QGL \prove F \leftrightarrow S(F)$.
	\end{lem}

	\begin{proof}
	Induction on the construction of $S(p)$.
		\begin{itemize}
		\item Assume $S(p) \equiv \Box A(p)$.
		Then $\QGL \prove S(\top) \leftrightarrow \left(\top \leftrightarrow S(\top) \right)$.
		By the derivation of $\QGL$, we have
			\begin{equation}
			\QGL \prove \Box S(\top) \leftrightarrow \Box \left( \top \leftrightarrow S(\top) \right). \label{E1}
			\end{equation} 
		Recall that $S(p)$ contains no occurrences of $p$ of depth $0$,
		and there is no variable which occurs freely in $S(\top)$ and is bounded in $S(p)$.
		By the substitution lemma,
			\begin{align*}
			\QGL \prove \Box \left( \top \leftrightarrow S(\top) \right)
			\to \left( S(\top) \leftrightarrow S(S(\top)) \right).
			\end{align*}
		From this and (\ref{E1}), we obtain
		$\QGL \prove \Box S(\top) \to \left( S(\top) \leftrightarrow S(S(\top)) \right)$.
		Since the formula $S(p)$ is a $\Sigma$-formula, so are $S(\top)$ and $S(S(\top))$.
		By Lemma \ref{L8}, $S(\top)$ and $S(S(\top))$ are self-provers.
		By Lemma \ref{L10},
		$\QGL \prove S(\top) \leftrightarrow S(S(\top))$.

		\item Assume $S(p) \equiv A(p) \land B(p)$, and let $F$ and $G$ be $\mathcal{L}''$-formulas such that
		$\QGL \prove F \leftrightarrow A(F)$ and $\QGL \prove G \leftrightarrow B(G)$.
		First, we have $\QGL \prove (F \land G) \to \left( F \leftrightarrow (F \land G) \right)$.
		By the derivation in $\QGL$, we get
			\begin{equation}
	 		\QGL \prove \Box (F \land G) \to \Box \left( F \leftrightarrow (F \land G) \right).  \label{E2}
			\end{equation}
		Note that all free variables occurring in $F$ (or $G$) are free variables occurring in $A(p)$
		(or $B(p)$, resp.). 
		By our supposition ($\dag$), no free variable occurring in $F$ or $F \land G$
		is bounded in $S(p)$, i.e., bounded in $A(p)$.
		By the substitution lemma,
			\begin{align*}
			\QGL \prove \Box (F \leftrightarrow F \land G) \to
			\left( A(F) \leftrightarrow A(F \land G) \right). 
			\end{align*}
		From this and (\ref{E2}),
		$\QGL \prove \Box (F \land G) \to \left( A(F) \leftrightarrow A(F \land G) \right)$.
		By $\QGL \prove F \leftrightarrow A(F)$, we obtain
		$\QGL \prove \Box (F \land G) \to \left( F \leftrightarrow A(F \land G) \right)$. 
		Similarly, we can derive
		$\QGL \prove \Box (F \land G) \to \left( G \leftrightarrow B(F \land G) \right)$. Thus,
		$\QGL \prove \Box (F \land G) \to
		\left( F \land G \leftrightarrow A(F\land G) \land B(F\land G)\right)$, i.e.,
		$\QGL \prove \Box (F \land G) \to \left( F \land G \leftrightarrow S(F\land G) \right)$.

		We claim that $F$ and $G$ are self-provers. We show this only for $F$.
		Since $A(F)$ is a $\Sigma$-formula, by Lemma \ref{L8}, $A(F)$ is a self-prover, 
		and hence $\QGL \prove A(F) \to \Box A(F)$. By the induction hypothesis,
		$\QGL \prove F \leftrightarrow A(F)$, and hence $\QGL \prove \Box F \leftrightarrow \Box A(F)$.
		Thus $\QGL \prove F \to \Box F$.
		
		By Lemma \ref{L8}, $F \land G$ is a self-prover. Since $S(p)$ is a $\Sigma$-formula,
		and so is $S(F\land G)$. By Lemma \ref{L8}, $S(F\land G)$ is a self-prover.
		By Lemma \ref{L10}, $\QGL \prove F \land G \leftrightarrow S(F\land G)$.

		\item Assume $S(p) \equiv A(p) \lor B(p)$, and let $F$ and $G$ be $\mathcal{L}''$-formulas such that
		$\QGL \prove F \leftrightarrow A(F)$ and $\QGL \prove G \leftrightarrow B(G)$.
		First, we have $\QGL \prove F \to ( F \leftrightarrow F \lor G)$. Then
			\begin{equation}
			\QGL \prove \Box F \to \Box ( F \leftrightarrow F \lor G ). \label{E3}
			\end{equation}
		Note that all free variables occurring in $F$ (or $G$) are free variables occurring in $A(p)$
		(or $B(p)$, resp.). 
		By our supposition ($\dag$), every free variable occurring in $F$ or $F \lor G$
		is not bounded in $S(p)$, i.e., not bounded in $A(p)$.
		By the substitution lemma,
			\begin{equation*}
			\QKf \prove \Box (F \leftrightarrow F \lor G )
			\to \left( A(F) \leftrightarrow A(F \lor G) \right). 
			\end{equation*}
		From this and (\ref{E3}),
		$\QKf \prove \Box F \to \left( A(F) \leftrightarrow A(F \lor G) \right)$. By the induction hypothesis,
		$\QGL \prove \Box F \to \left( F \leftrightarrow A(F \lor G) \right)$.
		Note that $F$ and $A(F \lor G)$ are self-provers. By Lemma \ref{L10},
		$\QGL \prove F \leftrightarrow A(F \lor G)$. Similarly, we can derive
		$\QGL \prove G \leftrightarrow B(F \lor G)$. Thus
		$\QGL \prove F \lor G \leftrightarrow A(F \lor G) \lor B(F \lor G)$, i.e.,
		$\QGL \prove F \lor G \leftrightarrow S(F \lor G)$.

		\item Assume $S(p) \equiv \exists u A(u)$, and let $F$ be an $\mathcal{L}''$-formula such that
		$\QGL \prove F \leftrightarrow A(F)$. Since $\QGL \prove F \to (F \leftrightarrow \exists u F)$, 
		we have $\QGL \prove \Box F \to \Box (F \leftrightarrow \exists u F)$.
		Note that no free variable occurring in $F$ or $\exists u F$ is bounded in $A(p)$.
		By the substitution lemma,
		$\QGL \prove \Box F \to \left( A(F) \leftrightarrow A(\exists u F) \right)$.
		By the induction hypothesis,
		$\QGL \prove \Box F \to \left( F \leftrightarrow A(\exists u F) \right)$.
		Recall that $F$ and $\exists u F$ are self-provers. By Lemma \ref{L10},
		$\QGL \prove F \leftrightarrow A(\exists u F)$, and hence
		$\QGL \prove \exists u F \leftrightarrow \exists u A(\exists u F)$,
		i.e., $\QGL \prove \exists u F \leftrightarrow S( \exists u F)$.
		\end{itemize}
	\end{proof}

	\begin{lem}\label{L12}
	For any $\Sigma$-formulas $S_0(p_0, \ldots, p_n), \ldots, S_n(p_0, \ldots, p_n)$,
	there are $\mathcal{L}''$-formulas $F_0, \ldots, F_n$ satisfying the desired properties such that for any $i \leq n$,
	$\QGL \prove F_i \leftrightarrow S_i (F_0, \ldots, F_n)$.
	\end{lem}

	\begin{proof}
	We prove Lemma by the induction on $n$. If $n=0$, then it follows from Lemma \ref{L11}.
 
	Suppose that Lemma holds for $\leq n$. Let $S_0(p_0, \ldots, p_{n+1}), \ldots, S_{n+1}(p_0, \ldots, p_{n+1})$
	be $\Sigma$-formulas. By the induction hypothesis, there are $\mathcal{L}''$-formulas
		\begin{gather*}
		F_0(p_{n+1}), \ldots, F_n(p_{n+1})
		\end{gather*}
	such that for any $i \leq n$,
	$\QGL \prove F_i(p_{n+1}) \leftrightarrow S_i \left(F_0(p_{n+1}), \ldots, F_n(p_{n+1}), p_{n+1} \right)$.
	Let $F$ be an $\mathcal{L}'$-formula such that
	$\QGL \prove F \leftrightarrow S_{n+1} \left(F_0(F), \ldots, F_{n}(F), F \right)$.
	(The existence of such an $F$ is guaranteed by Lemma \ref{L11}.)
	Then for any $i \leq n$,
	$\QGL \prove F_i(F) \leftrightarrow S_i \left(F_0(F), \ldots, F_n (F), F \right)$.
	Therefore, $\langle F_0(F), \ldots, F_n(F), F \rangle$ are desired formulas. The proof of the case $n+1$ is completed.
	\end{proof}

Finally, we prove Theorem \ref{T4}. 

\begin{proof}[Proof of Theorem \ref{T4}]
Let $A(p)$ be a Boolean combination of $\Sigma$-formulas and formulas containing no occurrences of $p$. Then there are a propositional formula $B(q_0, \ldots, q_{n-1}, r_0, \ldots, r_{m-1}) $, $\Sigma$-formulas $S_0(p), \ldots, S_{n-1}(p)$, and $\mathcal{L}''$-formulas $R_0, \ldots, R_{m-1}$ containing no occurrences of $p$, such that
	\begin{equation*}
	A(p) \equiv B \left( S_0(p), \ldots, S_{n-1}(p), R_0, \ldots, R_{m-1} \right).
	\end{equation*}
For each $i < n$, put $C_i (q_0, \ldots, q_{n-1}) :\equiv S_i \left( B(q_0, \ldots, q_{n-1}, R_0, \ldots, R_{m-1}) \right)$. By Lemma \ref{L12}, there are $F_0, \ldots, F_{n-1}$ such that for each $i < n$,
$\QGL \prove F_i \leftrightarrow C_i \left( F_0, \ldots, F_{n-1} \right)$.
Let $F :\equiv B( F_0, \ldots, F_{n-1}, R_0, \ldots, R_{m-1})$. Then we have $\QGL \prove F_i \leftrightarrow S_i(F)$, and hence $\QGL \prove F \leftrightarrow B \left( S_0(F), \ldots, S_{n-1}(F), R_0, \ldots, R_{m-1} \right)$, i.e., $\QGL \prove F \leftrightarrow A(F)$.

\end{proof}

\begin{prob}
Is there a formula $A(p)$ satisfying the following conditions?
	\begin{itemize}
	\item $A(p)$ is modalized in $p$;
	\item $A(p)$ is not provably equivalent to any Boolean combination of $\Sigma$-formulas
	and formulas containing no occurrences of $p$:
	\item $A(p)$ has a fixed-point in $\QGL$.
	\end{itemize}
\end{prob}

\end{document}